\documentclass{article}
\usepackage{amssymb,amsmath,amsfonts, amsthm, amscd}
\usepackage{arydshln}
\usepackage{float}
\usepackage{graphicx}
\usepackage{caption}
\usepackage{rotating}
\usepackage{pdflscape}
\usepackage{tikz}
\usepackage[utf8]{inputenc}
\usepackage{booktabs}
\usepackage{lipsum}
\usepackage{multirow}
\usepackage{enumitem}
\usepackage{subcaption}
\usepackage{multirow}
\usepackage{epsfig,amsthm,amsmath,amsfonts,amssymb}
\usepackage{fancyhdr,multicol,color,xcolor}
\usepackage[english]{babel}
\usepackage[utf8]{inputenc}
\usepackage{algorithm}
\usepackage[noend]{algpseudocode}
\numberwithin{equation}{section} \numberwithin{equation}{section}
 
\newcommand{\RNum}[1]{\uppercase\expandafter{\romannumeral #1\relax}}
\label{sec.sub.gdn}

\newtheorem{definition}{Definition}

\newtheorem{ex}{Example}
\newtheorem{theorem}{Theorem}

\newtheorem{remark}{Remark}

\numberwithin{equation}{section}
\begin{document}
	\title{\bf  On the Maximal Solution of A Linear System over Tropical Semirings}
	\author{{Sedighe Jamshidvand\textsuperscript{$a$}, Shaban Ghalandarzadeh\textsuperscript{$a$}, Amirhossein Amiraslani\textsuperscript{$b,a$}, Fateme Olia\textsuperscript{$a$}}\\
		{\em \small   $^a$ Faculty of Mathematics, K. N. Toosi University of
			Technology, Tehran, Iran} \\
		{\em \small   $^b$ STEM Department, The University of Hawaii-Maui College, Kahului, Hawaii, USA }}
	\maketitle{}		
	
\begin{abstract}
	In this paper, we present methods for solving a system of linear equations, $ AX=b $, over tropical semirings. To this end, if possible, we first reduce the order of the system through some row-column analysis, and obtain a new system with fewer equations and variables. We then use the pseudo-inverse of the system matrix to solve the system if solutions exist. Moreover, we propose a new version of Cramer's rule to determine the maximal solution of the system. Maple procedures for computing the pseudo-inverse are included as well.
\end{abstract}

	{\small {\it key words}: Tropical semiring; System of linear equations;  Cramer's rule; Pseudo-inverse}\\[0.3cm]
	\section{Introduction}
	Solving systems of linear equations is an important aspect of linear algebra. Considering solution techniques for systems of linear equations over rings and, as a special case, over fields, we intend to develop systematic methods to understand the behavior of linear systems and extend some well-known results over rings to tropical semirings. Systems of linear equations over tropical semirings find applications in various areas of engineering, computer science, optimization theory, control theory, etc~(see \cite{A1}, \cite{G1}, \cite{K1}, \cite{K2}, \cite{M1} ).
	
	Semirings are algebraic structures similar to rings, but subtraction and division can not necessarily be defined for them. The notion of a semiring was first introduced by Vandiver \cite{V1} in 1934. A semiring $(S,+,.,0,1)$ is an algebraic structure in which $(S, +)$ is a commutative monoid with an identity element $0$ and $(S,.)$ is a monoid with an identity element 1, connected by ring-like distributivity. The additive identity $0$ is multiplicatively absorbing, and $0 \neq 1$. Note that for convenience, we mainly consider $S=(\mathbb{R} \cup \{-\infty\}, max, +, -\infty, 0)$, which is a well-known tropical semiring called ``$\max-\rm plus$ algebra", in this work.  Other examples of tropical semirings, which are isomorphic to ``$\max-\rm plus$ algebra", are ``$\max-\rm times$ algebra", ``$\min-\rm times$ algebra "and ``$\min-\rm plus$ algebra".
	
	Suppose $ S$ is a tropical semiring. We want to solve the system of linear equations $ AX=b $, where $ A \in M_{n}(S),~ b \in S^{n} $ and $ X $ is an unknown vector. Sometimes, we can reduce the order of the system to obtain a new system with fewer equations, which is called a reduced system.
	Furthermore, we find necessary and sufficient conditions based on the ``pseudo-inverse'', $A^{-}$, of matrix $ A $ with determinant $ det_{\varepsilon}(A) \in U(S) $ to solve the system, where $ U(S) $ is the set of the unit elements of $ S $. This method is not limited to square matrices, and can be extended to arbitrary matrices of size $ m \times n $ as well. To this end, we must consider a square system of size $ \min\{ m,n \} $ corresponding to the non-square system.
	
	Cramer's rule is one of the most useful methods for solving a linear system of equations over rings and fields. As an extension of this method, we will establish a new version of Cramer's rule to obtain the ``maximal'' solution of a system of linear equations over tropical semirings. In~\cite{C1}, Cramer's rule over max-plus algebra has been fairly exhaustively studied. In this work, we propose an approach that is more aligned with Cramer's rule in classic linear algebra. We derive clear cut conditions based on the entries of $AA^{-}$ on when the method is applicable and yields a solution. Additionally, the solution that is found here is guaranteed to be the maximal solution of the system.    
	
	Through row-column analysis, we can remove linearly dependent rows and columns of the system matrix. As such, the reduced system is obtained with fewer equations and unknowns and it is proved that the existence of solutions of the primary system and the corresponding reduced system depend on each other; see Section~3 for more details.
	
	In Section~4, assuming that $ det_{\varepsilon}(A) \in U(S) $, we focus on methods for solving the square system $ AX=b $. Finding necessary and sufficient conditions on the entries of matrix $ AA^{-} $, we obtain the maximal solution $ X^{*} $ of the system $ AX=b $ using two methods: by computing the pseudo-inverse of $ A $, as $ A^{-}b $, and without computing the pseudo-inverse of $ A $ through Cramer's rule. Additionally, for solving the non-square system $ AX=b $, we construct a corresponding square system. If there are fewer equations than unknowns, then the solvability of the corresponding square system implies that the system $ AX=b $ should be solvable, too. If, however, there are more equations than unknowns, we can establish a square system simply by multiplying the transpose of the matrix $ A $ from left by the non-square system $ AX=b $. Note that in this case, the solution of the square system can not necessarily be a solution of the system $ AX=b $, unless it satisfies some conditions with regards to the eigenvalues and eigenvectors of the associated square matrix.
	
	In the appendix of this paper, we give some Maple procedures as follows. Table $1$ is a subroutine for finding the determinant of a square matrix. Table~$2$ gives a code for calculating matrix multiplications. Table $3$ consists of a program for finding the pseudo-inverse of a square matrix and computing the multiplication of the matrix by its pseudo-inverse. These procedures are all written for max-plus algebra.
	
	\section{Definitions and preliminaries}
	In this section, we give some definitions and preliminaries. For convenience, we use $  \mathbb{N} $ and $\underline{n}$ to denote the set of all positive integers and the set $\{ 1,2,\cdots,n\}$ for $n \in \mathbb{N}$, respectively.
	\begin{definition}
		A semiring $ (S,+,., 0,1)$ is an algebraic system consisting of a nonempty set $ S $ with two binary operations,  addition and multiplication, such that the following conditions hold:
		\begin{enumerate} 
			\item{$ (S, +) $ is a commutative monoid with identity element $ 0 $};
			\item{ $ (S, \cdot) $ is a monoid with identity element $ 1 $};
			\item{ Multiplication distributes over addition from either side, that is $ a(b+c)= ab+ac $ and $ (b+c)a=ba+ca $ for all $ a, b \in S $};
			\item{ The neutral element of $ S $ is an absorbing element, that is $ a\cdot 0 =0= 0 \cdot a  $ for all $ a \in S $};
			\(\)			\item{ $ 1 \neq 0 $}.
		\end{enumerate}
		A semiring is called commutative if $ a\cdot b = b \cdot a $ for all $ a, b \in S $.
	\end{definition}
	\begin{definition}
		A commutative semiring $(S,+,.,0,1)$ is called a semifield if every nonzero element of $S$ is multiplicatively invetrible.
	\end{definition}
	In this work our main focus is on tropical semiring $(\mathbb{R}\cup \{ -\infty \}, \max, +, -\infty, 0 )$, denoted by $\mathbb{R}_{\max, +}$, that is called ``$\max-\rm plus$ algebra" whose additive and multiplicative identities are $-\infty$ and $0$, respectively. Moreover, the notation $a-b$ in ``$\max-\rm plus$ algebra" is equivalent to $ a + (-b)  $, where $ ``-" ,~``+" $ and $ -b $ denote the usual real numbers subtraction, addition and the typical additively inverse of the element $ b $, respectively. Note further that ``$\max-\rm plus$ algebra" is a commutative semifield.
	\begin{definition}(See \cite{G1})
		Let $ S $ be a semiring. A left $ S $-semimodule is a commutative monoid $ (\mathcal{M},+) $ with identity element $ 0_{\mathcal{M}}$ for which we have a scalar multiplication function $S \times \mathcal{M} \longrightarrow \mathcal{M} $, denoted by $ (s, m)\mapsto sm $, which satisfies the following conditions for all $ s, s^{\prime} \in S $ and $ m, m^{\prime} \in \mathcal{M} $:
		\begin{enumerate}
			\item{$ (ss^{\prime})m=s(s^{\prime}m) $ };
			\item{$ s(m+m^{\prime})= sm +sm^{\prime}$};
			\item{$ (s+s^{\prime})m=sm +s^{\prime}m $};
			\item{$ 1_{S}m=m$};
			\item{$ s0_{\mathcal{M}}=0_{\mathcal{M}}=0_{S}m$.}
		\end{enumerate}  
	\end{definition}
	Right semimodules over $ S $ are defined in an analogous manner.
	\begin{definition}
		A nonempty subset $ \mathcal{N} $ of a left $ S$-semimodule $ \mathcal{M} $ is a subsemimodule of $ \mathcal{M} $ if $ \mathcal{N} $ is closed under addition and scalar multiplication. Note that this implies $ 0_{\mathcal{M}} \in \mathcal{N} $. Subsemimodules of right semimodules are defined analogously. 
	\end{definition}
	\begin{definition}
		Let $ \mathcal{M} $ be a left $ S $-semimodule and $ \{ \mathcal{N}_{i} \vert i \in \Omega \} $ be a family of subsemimodules of $ \mathcal{M} $. Then $ \displaystyle{\bigcap_{i \in \Omega}}\mathcal{N}_{i}$ is a subsemimodule of $ \mathcal{M} $ which, indeed, is the largest subsemimodule of $ \mathcal{M} $ contained in each of the $ \mathcal{N}_{i}$. In particular, if $ \mathcal{A} $ is a subset of a left $ S $-semimodule $ \mathcal{M} $, then the intersection of all subsemimodules of $ \mathcal{M} $ containing $ \mathcal{A} $ is a subsemimodule of $ \mathcal{M} $, called the subsemimodule generated by $ \mathcal{A} $. This subsemimodule is denoted by
		$$ S\mathcal{A}=Span(\mathcal{A}) = \{\displaystyle{\sum_{i=1}^{n}}s_{i}\alpha_{i}~ \vert~ s_{i} \in S, \alpha_{i} \in \mathcal{A},i \in \underline{n}, n \in \mathbb{N}  \}.$$
		If $  \mathcal{A}$ generates all of the semimodule $ \mathcal{M} $ , then $\mathcal{A} $ is a set of generators for $ \mathcal{M} $. Any set of generators for $ \mathcal{M} $ contains a minimal set of generators. A left $ S $-semimodule having a finite set of generators is finitely generated. Note that the expression $ \displaystyle{\sum_{i=1}^{n}}s_{i}\alpha_{i} $ is a linear combination of the elements of $ \mathcal{A} $.
	\end{definition}
	\begin{definition}(See \cite{T2})
		Let $\mathcal{M}$ be a left $ S $-semimodule. A nonempty subset, $ \mathcal{A}$, of $\mathcal{M} $ is called linearly independent if $ \alpha \notin Span(\mathcal{A} \setminus \{ \alpha \}) $ for any $ \alpha \in \mathcal{A} $. If $\mathcal{A}$ is not linearly independent, then it is called linearly dependent.
	\end{definition}
	\begin{definition}
		The rank of a left $ S $-semimodule $ \mathcal{M} $ is the smallest $ n $ for which there exists a set of generators  of $ \mathcal{M} $ with cardinality $ n $. It is clear that  $ rank(\mathcal{M})$ exists for any finitely generated left $ S $-semimodule $ \mathcal{M} $. \\
		This rank need not be the same as the cardinality of a minimal set of generators for $ \mathcal{M} $, as the following example shows. 
	\end{definition}
	\begin{ex}
		Let $S $ be a semiring and $\mathcal{R} =S \times S $ be the Cartesian product of two copies of $ S $. Then $ \{ (1_{S},1_{S}) \}$ and $ \{ (1_{S}, 0_{S}), (0_{S},1_{S}) \}$ are both minimal sets of generators for $ \mathcal{R} $, considered as a left semimodule over itself with componentwise addition and multiplication. Hence, $ rank(\mathcal{R})= 1 $.
	\end{ex}
	Let $S$ be a commutative semiring. We denote the set of all $m \times n$ matrices over $ S $ by $M_{m \times n}(S)$. For $A \in M_{m \times n} (S)$, we denote by $a_{ij}$ and $A^{T}$ the $(i,j)$-entry of $A$ and the transpose of $A$, respectively.\\
	For any $A, B \in M_{m \times n}(S)$, $C \in M_{n \times l}(S)$ and $\lambda \in S$, we define: $$A+B = (a_{ij}+b_{ij})_{m \times n},$$ $$AC=(\sum_{k=1}^{n} a_{ik}b_{kj})_{m \times l},$$ and $$\lambda A=(\lambda a_{ij})_{m\times n}.$$
	Clearly, $ M_{m \times n}(S) $ equipped with matrix addition and matrix scalar multiplication is a left $ S $-semimodule.
	It is easy to verify that $M_{n}(S):=M_{n \times n}(S)$ forms a semiring with respect to the matrix addition and the matrix multiplication.\\
	The above matrix operations over $ \max-\rm plus$ algebra can be considered as follows.
	$$A+B = (\max(a_{ij},b_{ij}))_{m \times n},$$ 
	$$AC=(\max_{k=1}^{n} (a_{ik}+b_{kj}))_{m \times l},$$ 
	and $$\lambda A=(\lambda + a_{ij})_{m\times n}.$$
	For convenience, we can denote the scalar multiplication $ \lambda A $ by $ \lambda + A $. Moreover, $ \max-\rm plus$ algebra is a commutative semiring which implies $ \lambda + A= A +\lambda $.
	\begin{definition}
		Let $A, B \in M_{n}(S)$ such that $A=(a_{ij})$ and $B=(b_{ij})$. We say $A\leq B$ if and only if $a_{ij} \leq b_{ij}$ for every $i\in \underline{m}$ and $j \in \underline{n}$. 
	\end{definition}
	Let $A \in M_{n}(S)$, $\mathcal{S}_{n}$ be the symmetric group of degree $n \geq 2$ , and $\mathcal{A}_{n}$ be the alternating group on $n$ such that
	\begin{center}
		$\mathcal{A}_{n}= \{ \sigma \vert \sigma \in \mathcal{S}_{n} ~\text{and}~\sigma \text{~is~an~even~permutation}\}$. 
	\end{center} 
	The positive determinant, $\vert A \vert ^{+}$, and negative determinant, $\vert A \vert ^{-}$, of $A$ are 
	\begin{center}
		$$\displaystyle{\vert A \vert ^{+}=\sum_{\sigma \in \mathcal{A}_{n}} \prod_{i=1}^{n} a_{i\sigma (i)}},$$
	\end{center}
	and
	\begin{center}
		$$\displaystyle{\vert A \vert ^{-}= \sum_{\sigma \in \mathcal{S}_{n} \backslash \mathcal{A}_{n}} \prod_{i=1}^{n} a_{i\sigma (i)}}. $$
	\end{center}
	Clearly, if $S$ is a commutative ring, then $\vert A \vert=\vert A \vert ^{+} - \vert A \vert ^{-}$.
	\begin{definition}
		Let $S$ be a semiring. A bijection $\varepsilon$ on $S$ is called an $\varepsilon$-function of $S$ if $\varepsilon(\varepsilon(a))=a$, $\varepsilon(a+b)= \varepsilon(a)+\varepsilon(b)$, and $\varepsilon(ab)=\varepsilon(a)b=a\varepsilon(b)$ for all $a, b \in S$. Consequently, $\varepsilon(a)\varepsilon(b)=ab$ and $\varepsilon(0)=0$.\\
		The identity mapping: $a \mapsto a$ is an $\varepsilon$-function of $S$ that is called the identity $\varepsilon$-function.
	\end{definition}
	\begin{remark}
		Any semiring $S$ has at least one $\varepsilon$-function since the identical mapping of $S$ is an $\varepsilon$-function of $S$. If $S$ is a ring, then the mapping : $a \mapsto -a$, $( a \in S)$ is an $\varepsilon$-function of $S$.
	\end{remark}
	\begin{definition}
		Let $S$ be a commutative semiring with an $\varepsilon$-function, $\varepsilon$, and $A \in M_{n}(S)$. The $\varepsilon$-determinant of $A$, denoted by $det_{\varepsilon}(A)$, is defined by 
		\begin{center}
			$$det_{\varepsilon}(A)= \displaystyle{\sum_{\sigma \in \mathcal{S}_{n}}\varepsilon ^{\tau(\sigma)}(a_{1\sigma(1)}a_{2\sigma(2)}\cdots a_{n\sigma(n)})}$$
		\end{center}
		where $\tau(\sigma)$ is the number of the inversions of the permutation $\sigma$, and $\varepsilon^{(k)}$ is defined by $\varepsilon^{(0)}(a)=a$ and $\displaystyle{\varepsilon^{(k)}(a)=\varepsilon^{k-1}(\varepsilon(a))}$ for all positive integers $k$.
		Since $\varepsilon^{(2)}(a)=a$, $det_{\epsilon}(A)$ can be rewritten in the form of
		\begin{center}
			$det_{\epsilon}(A)=\vert A \vert ^{+}+\varepsilon(\vert A \vert^{-})$.
		\end{center}
		In particular, for $S=\mathbb{R}_{\max, +}$ with the identity $\varepsilon$-function, we  have $det_{\varepsilon}(A)=\max(\vert A \vert ^{+}, \vert A \vert ^{-})$.
	\end{definition}
	\begin{definition}
		Let $S$ be a commutative semiring with $\varepsilon$-function, $\varepsilon$, and $A \in M_{n}(S)$. The $\varepsilon$-adjoint matrix $A$, written as $adj_{\varepsilon}(A)$, is defined as follows.\\
		\begin{center}
			$adj_{\varepsilon}(A)=((\varepsilon^{(i+j)}det_{\varepsilon}(A(i | j)))_{n \times n})^{T}$,
		\end{center}
		where $ A(i|j) $ denotes the $ (n-1) \times (n-1) $ submatrix of $ A $ obtained from $ A $ by removing the $ i $-th row and the $ j $-th column. It is clear that if $S$ is a commutative ring, $\varepsilon$ is the mapping: $a \mapsto -a$;$(a \in S)$, and $A \in M_{n}(S)$, then $adj_{\varepsilon}(A)=adj(A)$.
	\end{definition}
	\begin{theorem} (See \cite{T1})\label{r-c} Let $A \in M_{n}(S)$. We have 
		\begin{enumerate}
			\item $Aadj_{\varepsilon}(A)=(det_{\varepsilon}(A_{r}(i \Rightarrow j)))_{n \times n}$,
			\item $adj_{\varepsilon}(A)A=(det_{\varepsilon}(A_{c}(i \Rightarrow j)))_{n \times n}$,
		\end{enumerate}
		where $A_{r}(i \Rightarrow j)$ ($A_{c}(i \Rightarrow j)$) denotes the matrix obtained from $A$  by replacing the $j$-th row (column) of $A$ by the $i$-th row (column) of $A$.
	\end{theorem}
	\begin{definition}
		(See \cite{W1})
		Let $ S $ be a semiring and $A \in M_{m \times n}(S) $. The column space of $ A $ is the finitely generated right $ S $-subsemimodule of $M_{m \times 1}(S) $ generated by the columns of $A$:
		$$ Col(A)=\{ Av \vert v \in M_{n \times 1}(S) \}. $$
		The column rank of $A $ is the rank of its column subsemimodule, which is denoted by $colrank(A)$.
	\end{definition}
	\begin{definition}
		(See\cite{W1})
		Let $ S $ be a semiring and $A \in M_{m \times n}(S) $. The row space of $ A $ is the finitely generated left $ S $-subsemimodule of $M_{1 \times n}(S) $ generated by the rows of $A$:
		$$ Row(A)=\{ uA \vert u \in M_{1 \times m}(S) \}. $$
		The row rank of $A $ denoted by $rowrank(A)$ is the rank of its row subsemimodule.
	\end{definition}
	The next example shows that the column rank and the row rank of a matrix over an arbitrary semiring are not necessarily equal. If these two value  coincide, their common value is called the rank of matrix $ A $.
	\begin{ex}
		Consider $A \in M_{3}(S) $ where $ S=\mathbb{R}_{\max,+}$ as follows.\\
		\[
		A=\left[
		\begin{array}{ccc}
		3&6&5\\
		-5&0&-2\\
		4&1&6
		\end{array}
		\right].
		\]
		Clearly, $ rowrank(A)=3 $, but $colrank(A)=2$, since  the third column of $ A $ is a linear combination of its other columns:
		$$ C_{3}=\max(C_{1}+ 2, C_{2}+(-2)).$$
	\end{ex}
	Next, we study and analyze the system of linear equations $ AX=b $ where $A \in M_{m\times n}(S) $, $ b \in S^{m} $ and  $ X $ is an unknown column vector of size $n$ over tropical semiring $ S=\mathbb{R}_{\max,+}$, whose $i-$th equation is
	$$\max(a_{i1}+ x_{1}, a_{i2} +x_{2}, \cdots, a_{in} +x_{n})=b_{i}. $$
	Sometimes, we can simplify the solution process of the system, $ AX=b $, by turning that into a linear system of equations with fewer equations and variables. 
	\begin{definition}
		Let $A \in M_{m \times n}(S) $. A reduced matrix is obtained from matrix $ A $ by removing its  dependent rows and columns which we denote by $ \overline{A}$. 
	\end{definition}
	\begin{definition}
		A solution $X^{*}$ of the system $AX=b$ is called maximal if $X \leq X^{*}$ for any solution $X$.
	\end{definition}
	\begin{definition}
		Let $b\in S^{m}$. Then $ b $ is called a regular vector if 
		$ b_{i} \neq -\infty $ for any $i \in \underline{m}$.
	\end{definition}
	Without loss of generality, we can assume that $b$ is regular in the system $AX=b$. Otherwise, let $b_{i}=-\infty $ for some $i \in \underline{n}$. Then in the $i-$th equation of the system, we have $a_{ij} + x_{j}= -\infty $ for any $j \in \underline{n}$. As such, $x_{j}= -\infty $ if $a_{ij}\neq -\infty $. Consequently, the $i-$th equation can be removed from the system together with every column $A_{j}$ where $a_{ij} \neq -\infty $, and the corresponding $x_{j}$ can be set to $-\infty $.
	\begin{definition}
		Let $ A \in M_{n}(S)$, $ \lambda \in S $ and $ x \in S^{n} $ be a regular vector such that 
		$$ Ax=\lambda x.$$ 
		Then $\lambda $ is called an eigenvalue of $A $ and $ x $ an eigenvector of $ A $ associated with eigenvalue $\lambda $.
		Note that this definition allows an eigenvalue to be $ - \infty $. Moreover, eigenvectors are allowed to contain elements equal to $- \infty $.
	\end{definition}
	\section{Row-Column analysis of the system $ AX=b$ }
	In this section, we reduce the order of the system $AX=b$, where $A \in M_{m\times n}(S) $, $ b \in S^{m} $ and  $ X $ is an unknown column vector of size $n$, through a row-column analysis to obtain a new system with fewer equations and variables which is called a reduced system.\\
	\subsection{Column analysis of the system $ AX=b$}
	
	Suppose that $ C_{1},C_{2}, \cdots ,C_{n} $ are the columns of matrix $ A $. Without loss of generality, we can assume that $ C_{1},C_{2}, \cdots ,C_{k} $ are linearly independent and the other columns are linearly dependent on them.
	\begin{center}
		\[
		\left[
		\begin{array}{c|c|c|c|c|c|c}
		C_{1}&C_{2}&\cdots&C_{k}&C_{k+1}&\cdots&C_{n}
		\end{array}
		\right]
		\left[
		\begin{array}{c}
		x_{1}\\
		x_{2}\\
		\vdots\\
		x_{k}\\
		x_{k+1}\\
		\vdots\\
		x_{n}
		\end{array}
		\right]
		=
		\left[
		\begin{array}{c}
		b_{1}\\
		b_{2}\\
		\vdots\\
		b_{m}
		\end{array}
		\right].
		\]	
	\end{center}
	We can rewrite the system as follows.
	\begin{center}
		\begin{equation}
		\label{eq1}
		\max (C_{1} + x_{1},C_{2} + x_{2}, \cdots, C_{k} + x_{k}, C_{k+1} + x_{k+1}, \cdots, C_{n} + x_{n} ) =
		\left[
		\begin{array}{c}
		b_{1}\\
		b_{2}\\
		\vdots\\
		b_{m}
		\end{array}
		\right]. ~~~~~~~~~
		\end{equation}
	\end{center}
	There exist scalars $\eta_{ij} \in S $ for every $ 1\leq i\leq k $ and $ k+1\leq j\leq n $ such that 
	\begin{center}
		\begin{equation}
		\label{eq2}
		C_{j} = \max(C_{1} + \eta_{1j},C_{2} + \eta_{2j}, \cdots, C_{k} + \eta_{kj}).
		\end{equation}
	\end{center}
	By replacing $ (\ref{eq2}) $ in $(\ref{eq1})$ we have:\\
	\begin{center}
		$\max (C_{1} + x_{1},C_{2} + x_{2}, \cdots, C_{k} + x_{k}, \max(C_{1} + \eta_{1(k+1)},C_{2} + \eta_{2(k+1)}, \cdots, C_{k} + \eta_{k(k+1)}) + x_{k+1}, \cdots,  \max(C_{1} + \eta_{1n},C_{2} + \eta_{2n}, \cdots, C_{k} + \eta_{kn}) + x_{n} ) =
		\left[
		\begin{array}{c}
		b_{1}\\
		b_{2}\\
		\vdots\\
		b_{m}
		\end{array}
		\right].$
	\end{center}
	Due to the distributivity of $ ``+"$ over $`` \max" $, the following equality is obtained:\\
	\begin{center}
		$\max [C_{1} +\max(x_{1}, \eta_{1(k+1)} + x_{k+1},\cdots,\eta_{1n} + x_{n}),C_{2} +\max(x_{2}, \eta_{2(k+1)} + x_{k+1},\cdots,\eta_{2n} + x_{n}),\cdots,C_{k} +\max(x_{k}, \eta_{k(k+1)} + x_{k+1},\cdots,\eta_{kn} + x_{n})]=
		\left[
		\begin{array}{c}
		b_{1}\\
		b_{2}\\
		\vdots\\
		b_{m}
		\end{array}
		\right].$
	\end{center}
	Now, we can rewrite this system as
	\begin{center}
		$\max (C_{1} + y_{1},C_{2} + y_{2}, \cdots, C_{k} + y_{k} ) =
		\left[
		\begin{array}{c}
		b_{1}\\
		b_{2}\\
		\vdots\\
		b_{m}
		\end{array}
		\right],$
	\end{center}
	where 
	\begin{center}
		\begin{equation}
		\label{eq3}
		y_{i}= \max(x_{i}, \eta_{i(k+1)} + x_{k+1},\cdots,\eta_{in} + x_{n}),
		\end{equation}
	\end{center}
	for every $ 1\leq i\leq k $. As such, the number of variables decreases from $ n $ to $ k $.
	Next, we show that the existence of solutions of the system $ AX=b $ depends on the row rank of $ A $. Assume that $ Y^{*} =(y^{*}_{i})_{i=1}^{k} $ is the maximal solution of the system: 
	\begin{center}
		\[
		\left[
		\begin{array}{c|c|c|c}
		C_{1}&C_{2}&\cdots&C_{k}
		\end{array}
		\right]
		\left[
		\begin{array}{c}
		y^{*}_{1}\\
		y^{*}_{2}\\
		\vdots\\
		y^{*}_{k}
		\end{array}
		\right]
		=
		\left[
		\begin{array}{c}
		b_{1}\\
		b_{2}\\
		\vdots\\
		b_{m}
		\end{array}
		\right].
		\]
	\end{center}
	Hence, the equalities $ (\ref{eq3})$ imply the system $ AX=b $ should have solutions $ x_{j}\leq \min (y^{*}_{1} -\eta_{1j},y^{*}_{2} -\eta_{2j}, \cdots,y^{*}_{k} -\eta_{kj}) $ for every $ k+1 \leq j \leq n $ and $ x_{i} = y^{*}_{i}  $ for every $ 1 \leq i \leq k $.
	\subsection{Row analysis of the system $ AX =b $}
	Consider the system $ AX =b $ in the form of
	\begin{center}
		\begin{equation}
		\label{eq4}
		\left[
		\begin{array}{c}
		R_{1}\\ \hline
		R_{2}\\ \hline
		\vdots\\ \hline
		R_{h}\\ \hline
		R_{h+1}\\ \hline
		\vdots \\\hline
		R_{m}\\
		\end{array}
		\right]
		\left[
		\begin{array}{c}
		x_{1}\\
		x_{2}\\
		\vdots\\
		x_{n}
		\end{array}
		\right]
		=
		\left[
		\begin{array}{c}
		b_{1}\\
		b_{2}\\
		\vdots\\
		b_{h}\\
		b_{h+1}\\
		\vdots \\
		b_{m}
		\end{array}
		\right],
		\end{equation}
	\end{center}
	where $ R_{i} $ is the $ i $-th row of the matrix $ A $, for every $ 1 \leq i \leq m $. Without loss of generality, we can assume that $ R_{1}, R_{2}, \cdots, R_{h} $ are linearly independent rows of $A$ and the other rows $ R_{i} $, $h+1 \leq i \leq m $ are linear combinations of them. Consequently, there exist scalars $ \xi _{ij} \in S $ for every $ 1 \leq j \leq h $ and $ h+1 \leq i \leq m $ such that:
	\begin{center}
		\begin{equation}
		\label{eq5}
		R_{i}= \max (R_{1} + \xi_{i1}, R_{2} + \xi_{i2}, \cdots,R_{h} + \xi_{ih}),
		\end{equation}
	\end{center}
	for every $ h+1 \leq i \leq m$. We can now rewrite the system of equations $(\ref{eq4})$ as 
	\begin{center}
		$R_{i}
		\left[
		\begin{array}{c}
		x_{1}\\
		x_{2}\\
		\vdots\\
		x_{n}
		\end{array}
		\right]= b_{i} $, for any $ 1 \leq i \leq m$,
	\end{center} 
	which can become the $ h $-equation system:
	\begin{center}
		$R_{j}
		\left[
		\begin{array}{c}
		x_{1}\\
		x_{2}\\
		\vdots\\
		x_{n}
		\end{array}
		\right]= b_{j} $, for any $ 1 \leq j \leq h$.
	\end{center}
	We now obtain the row-reduced system with $h$ equations.\\
	Note that in the process of reducing the system $AX=b$, it does not matter which of the row or column analysis is first applied to the system.\\
	This argument leads us to investigate the existence of solutions of the linear system $ AX=b $. 
	\begin{theorem}\label{SEOVER}
		Let $A \in M_{m \times n}(S)$. The system $AX=b$ has solutions if and only if its reduced system, $\overline{A}Y=\overline{b}$, has solutions. 
	\end{theorem}
	
	\begin{proof}
		Let $ colrank(A)=k$ and $rowrank(A)=h $. By applying row-column analysis on the system $AX=b$ and replacing $(\ref{eq5})$ in the $ m $-equation system $(\ref{eq4})$, we conclude that
		\begin{center}
			\begin{equation}
			\label{eq6}
			b_{i}=
			R_{i}
			\left[
			\begin{array}{c}
			x_{1}\\
			x_{2}\\
			\vdots\\
			x_{n}
			\end{array}
			\right]= \max(b_{1} + \xi_{i1}, b_{2} + \xi_{i2}, \cdots,b_{h} + \xi_{ih}),
			\end{equation}
		\end{center}
		for every $ h+1 \leq i \leq m$. If the equalities $(\ref{eq6})$ hold for every $ h+1 \leq i \leq m $, then we can reduce the system $ AX=b $ to the system $\overline{A}Y=\overline{b}$, where $ \overline{A}$ is the reduced $ h\times k $ matrix obtained from $ A $, $Y $ is an unknown vector of size $ k $, and $ \overline{b}$ is the reduced vector obtained from $ b $. Thus, the existence of solution $AX=b$ and $\overline{A}Y=\overline{b}$ depends on each other.
	\end{proof}
	\begin{remark}
		Note that if $\overline{b}$ is not a linear combination of any column of $\overline{A}$, then the systems $AX=b$ and $\overline{A}Y=\overline{b}$ have no solutions.
	\end{remark}
	\section{Solution methods for a linear system}
	In this section, we present methods for solving a linear system of equations.
	\subsection{Solving a linear system of equations using the pseudo-inverse of the system matrix}
	
	\begin{definition}
		Let $A \in M_{n}(S)$ and $det_{\varepsilon}(A) \in U(S)$. The pseudo-inverse of $A$, denoted by $ A^{-}$, is defined as
		\begin{center}
			$A^{-}=	det_{\varepsilon}(A)^{-1} adj_{\varepsilon}(A)$.
		\end{center}
		Especially, if $S=\mathbb{R}_{max,+}$, then $A^{-}=(a_{ij}^{-})$ where $a_{ij}^{-}= (adj_{\varepsilon}(A))_{ij} - det_{\varepsilon}(A)$.
	\end{definition}
	\begin{remark}\label{AA^{-}rem}
		Let $ A\in M_{n}(S)$. Then $ AA^{-}=((AA^{-})_{ij}) $ is a square matrix of size $ n $, such that by Theorem~\ref{r-c},
		\begin{align*}
		(AA^{-})_{ij} &= det_{\varepsilon}(A)^{-1} Aadj_{\varepsilon}(A))_{ij}\\
		&= det_{\varepsilon}(A)^{-1} det_{\varepsilon}(A_{r}(i \Rightarrow j)).
		\end{align*}
		In max-plus algebra, this becomes
		\begin{align*}
		(AA^{-})_{ij} &=(Aadj_{\varepsilon}(A))_{ij} - det_{\varepsilon}(A)\\
		& =det_{\varepsilon}(A_{r}(i \Rightarrow j)) - det_{\varepsilon}(A).
		\end{align*}
		The matrix $ A^{-}A $ is defined similarly. Note further that the diagonal entries of the matrices $AA^{-}$ and $A^{-}A$ are $0 $:
		\begin{align*}
		(AA^{-})_{ii} &= (Aadj_{\varepsilon}(A))_{ii} - det_{\varepsilon}(A)\\
		& = det_{\varepsilon}(A_{r}(i \Rightarrow i))-  det_{\varepsilon}(A)\\
		& = det_{\varepsilon}(A) -det_{\varepsilon}(A)\\
		& =  0
		\end{align*}
	\end{remark}
	
	\begin{theorem}\label{psuedothm}
		Let $A \in M_{n}(S)$ and $b \in S^{n}$ be a regular vector. Then $(AA^{-})_{ij} \leq b_{i} - b_{j}$ for any $i,j \in \underline{n}$ if and only if the system $AX=b$ has the maximal solution $X^{*}=A^{-}b$ where $X^{*}=(x_{i}^{*})_{i=1}^{n}$.
	\end{theorem}
	\begin{proof}
		Suppose that $(AA^{-})_{ij} \leq b_{i} - b_{j}$ for any $i,j \in \underline{n}$. First, we show that the system $AX=b$ has the solution $X^{*}=A^{-}b$. Clearly, $AX^{*}=AA^{-}b$, so for any $i\in \underline{n}$:
		\begin{align*}
		(AX^{*})_{i} &= (AA^{-}b)_{i}\\
		& =\max_{j=1}^{n}((AA^{-})_{ij}+b_{j})\\
		& =\max ((AA^{-})_{ii}+b_{i}, \max_{i \neq j} ((AA^{-})_{ij}+b_{j})).
		\end{align*}
		Since for any $i, j \in \underline{n}$, $(AA^{-})_{ij}+ b_{j} \leq b_{i}$, and $(AA^{-})_{ii}+ b_{i} = b_{i} $, we have $(AX^{*})_{i}= b_{i}$. As such, $X^{*}$ is a solution of the system $AX=b$.\\
		Now, we prove $X^{*}$ is a maximal solution. $AX^{*}=b$, so $A^{-}AX^{*}=X^{*}$.  The $k$-th equation of the system $A^{-}AX^{*}=X^{*}$ is
		\begin{center}
			$	max ((A^{-}A)_{k1}+x^{*}_{1}, \cdots, x^{*}_{k}, \cdots, (A^{-}A)_{kn}+x^{*}_{n} )=x^{*}_{k},$
		\end{center}
		that implies 
		\begin{center}
			\begin{equation}
			\label{eq7}
			(A^{-}A)_{kl} \leq x^{*}_{k} - x^{*}_{l}.
			\end{equation}
		\end{center}
		for any $l \neq k$. Now, suppose that $Y=(y_{i})_{i=1}^{n}$ is another solution of the system $AX=b$. This means $AY=b$, and $(A^{-}A )Y=X^{*}$. Without loss of generality, we can assume there exists only $j \in \underline{n}$ such that $y_{j} \neq x_{j}^{*}$, i.e., $y_{i}=x_{i}^{*}$ for any $i \neq j$. The $j$-th equation of the $A^{-}AY=X^{*}$ is 
		\begin{center}
			$max ((A^{-}A)_{j1}+x^{*}_{1}, \cdots, (A^{-}A)_{jj}+y_{j}, \cdots, (A^{-}A)_{jn}+x^{*}_{n})=x^{*}_{j}$.
		\end{center}
		This means $(A^{-}A)_{jj}+y_{j} \leq x^{*}_{j}$ which implies $y_{j} < x_{j}^{*}$. Moreover, if all inequalities $(\ref{eq7})$  for $k=j$ are proper, then 
		\begin{center}
			$max((A^{-}A)_{j1}+x^{*}_{1}, \cdots, y_{j}, \cdots, (A^{-}A)_{jn}+x^{*}_{n}) < x^{*}_{j}$.
		\end{center}
		Hence, $Y$ is not the solution of the system $AX=b$. That leads to a contradiction.\\
		This happens if all inequalities in $(\ref{eq7})$ are proper, so we can conclude that $X^{*}$ is a unique solution of the system $AX=b$. Otherwise, if some of the inequalities are not proper, i.e., $(A^{-}A)_{jl}=x^{*}_{j}-x^{*}_{l}$ for some $l \neq j$, then $Y$ is a solution of the system $AX=b$ such that $Y \leq X^{*}$. Consequently, $X^{*}$ is a maximal solution.\\
		Conversely, suppose that $X^{*}=A^{-}b$ is a maximal solution of the system $AX=b$. Then $AA^{-}b=b$. That implies $(AA^{-})_{ij} \leq b_{i}- b_{j}$ for any $i, j \in \underline{n}$.
	\end{proof}
	In the following example, we show that $(AA^{-})_{ij} \leq b_{i} - b_{j}$ is a sufficient condition for the system $ AX=b $ to have the maximal solution $ X^{*}=A^{-}b $.
	\begin{ex}
		Let $ A\in M_{4}(S) $. Consider the following system $ AX=b $:	
		\[
		\left[
		\begin{array}{cccc}
		1&-6&2&-5\\
		4&5&1&-2\\
		7&-1&3&0\\
		-2&-9&-5&0
		\end{array}
		\right]
		\left[
		\begin{array}{c}
		x_{1}\\
		x_{2}\\
		x_{3}\\
		x_{4}
		\end{array}
		\right]
		=
		\left[
		\begin{array}{c}
		2\\
		7\\
		3\\
		-4
		\end{array}
		\right],
		\]
		where $det_{\varepsilon}(A)=14$. Due to Theorem~\ref{psuedothm}, we must check the condition $(AA^{-})_{ij} \leq b_{i} - b_{j}$ for any $i,j \in \{ 1,\cdots ,4 \} $ where 
		$ (AA^{-})_{ij}=(Aadj_{\varepsilon}(A))_{ij} - det_{\varepsilon}(A) = det_{\varepsilon}(A_{r}(i \Rightarrow j)) - det_{\varepsilon}(A) $ (see Theorem~\ref{r-c} ). As such, $ AA^{-} $ is
		\[
		\left[
		\begin{array}{cccc}
		0&-11&-6&-5\\
		-1&0&-3&-2\\
		1&-6&0&0\\
		-7&-14&-9&0\\
		\end{array}
		\right].
		\]
		Indeed, it is easier to check $ (AA^{-})_{ij} \leq b_{i} - b_{j} \leq -(AA^{-})_{ji} $ for any $ 1 \leq i\leq j \leq 4 $.\\
		Since these inequalities hold, for instance $ (AA^{-})_{12} \leq 2 - 7 \leq -(AA^{-})_{21} $, the system $ AX=b $ has the maximal solution $ X^{*}=A^{-}b$:
		\[
		X^{*}=\left[
		\begin{array}{cccc}
		-6&-13&-7&-7\\
		-6&-5&-8&-7\\
		-2&-13&-8&-7\\
		-7&-14&-9&0
		\end{array}
		\right]
		\left[
		\begin{array}{c}
		2\\
		7\\
		3\\
		-4
		\end{array}
		\right]
		=
		\left[
		\begin{array}{c}
		-4\\
		2\\
		0\\
		-4
		\end{array}
		\right].
		\]
	\end{ex}
	The next example shows that the condition of Theorem~\ref{psuedothm} is  necessary. 
	\begin{ex}
		Let $ A\in M_{4}(S) $. Consider the following system $ AX=b $:	
		\[
		\left[
		\begin{array}{cccc}
		5&2&8&10\\
		4&1&7&9\\
		3&7&9&11\\
		-1&0&2&4
		\end{array}
		\right]
		\left[
		\begin{array}{c}
		x_{1}\\
		x_{2}\\
		x_{3}\\
		x_{4}
		\end{array}
		\right]
		=
		\left[
		\begin{array}{c}
		7\\
		4\\
		8\\
		1
		\end{array}
		\right].
		\]
		Then the matrix $ AA^{-}$ is as follows.
		\[
		AA^{-}=\left[
		\begin{array}{cccc}
		0&1&-1&6\\
		-1&0&-2&5\\
		1&2&0&7\\
		-6&-5&-7&0\\
		\end{array}
		\right].
		\]
		It can be checked that $ (AA^{-})_{ij} \leq b_{i} - b_{j} \leq -(AA^{-})_{ji} $ do not hold for  $ i =2 $ or $j=2 $.
		Therefore, $ X^{*}=A^{-}b$ cannot be the solution of the system $ AX=b $, where $ X^{*} $ is 
		\[
		X^{*}=\left[
		\begin{array}{cccc}
		-5&-4&-6&1\\
		-6&-5&-7&0\\
		-8&-7&-9&-2\\
		-10&-9&-11&-4
		\end{array}
		\right]
		\left[
		\begin{array}{c}
		7\\
		4\\
		8\\
		1
		\end{array}
		\right]
		=
		\left[
		\begin{array}{c}
		2\\
		1\\
		-1\\
		-3
		\end{array}
		\right].
		\]
		If $ X^{*} $ is the solution of the system $ AX=b $, then in the second equation of the system we encounter a contradiction:
		$$ \max (a_{21}+x^{*}_{1}, a_{22}+x^{*}_{2}, a_{23}+x^{*}_{3}, a_{24}+x^{*}_{4})= 6\neq b_{2} .$$
	\end{ex}
	\subsubsection{\label{nonsquare}\textbf{Extension of the method to non-square linear systems}}
	We are interested in studying the solution of a non-square linear system of equations as well. Let $ A \in M_{m\times n}(S)$ with $ m \neq n $, and $ b \in S^{m} $ be a regular vector. For solving the non-square system $ AX=b $ by Theorem~\ref{psuedothm}, we must consider a square linear system of order $ min\{ m, n \}$ corresponding to it. Without loss of generality, we can assume $ A $ is a reduced matrix, i.e. the number of independent rows(columns) is m(n), respectively. Since $ m \neq n $, we have the following two cases:
	\begin{enumerate}
		\item If $ m < n $, then we consider the square linear system of order $ m $ corresponding to the system $AX=b$. Let $ X=A^{T}Y $ where $ Y $ is an unknown vector of size $ m $. Then the square linear system $ AA^{T}Y=b $ is obtained from replacing $ X $ in $ AX=b $. Suppose that the conditions of Theorem~\ref{psuedothm} hold for the system $ AA^{T}Y=b $, so the system $ AA^{T}Y=b $ has the maximal solution $ Y^{*}=(AA^{T})^{-}b $. If so, the system $ AX=b $ has (at least) a solution in the form of $ X = A^{T}Y^{*} = A^{T}(AA^{T})^{-}b $, which is not necessarily maximal. The matrix $ A^{T}(AA^{T})^{-} $, denoted by $ A^{\dagger}$, is called the semi-psuedo-inverse of matrix $ A $. Hence, the system $ AX=b $ has the solution $ X=A^{\dagger}b $.
		\item If $ n < m $, then we consider the square linear system of size $ n $ corresponding to the system $ AX=b $. Clearly, we have the square linear system $ A^{T}AX=A^{T}b $ of size $ n $. Assume that the conditions of Theorem~\ref{psuedothm} hold for the system $ A^{T}AX=A^{T}b $. If so, it has the maximal solution $ X^{*}= (A^{T}A)^{-}A^{T}b $. Note further that $ X^{*} =(A^{T}A)^{-}A^{T}b $ is not necessarily the solution of the system $ AX=b $ unless $ b $ is an eigenvector of $ A(A^{T}A)^{-}A^{T} $ corresponding to the eigenvalue $ 0 $, i.e.; $AX^{*} = A(A^{T}A)^{-}A^{T}b = b $.
	\end{enumerate}		
	In the next examples, we try to solve some non-square linear systems if possible.
	\begin{ex}
		Let $ A\in M_{4 \times 5}(S) $. Consider the following system $ AX=b $:
		\[
		\left[
		\begin{array}{ccccc}
		-4&7&12&-3&0\\
		3&2&8&3&-1\\
		-9&1&6&0&2\\
		2&8&-5&1&-3
		\end{array}
		\right]
		\left[
		\begin{array}{c}
		x_{1}\\
		x_{2}\\
		x_{3}\\
		x_{4}\\
		x_{5}
		\end{array}
		\right]=
		\left[
		\begin{array}{c}
		14\\
		10\\
		8\\
		11
		\end{array}
		\right].	
		\]
		Due to the extension method
		, the non-square system $ AX=b $ can be converted into the following square system $ AA^{T}Y=b $, cosidering $ X=A^{T}Y $:
		\[
		\left[
		\begin{array}{cccc}
		24&20&18&15\\
		20&16&14&10\\
		18&14&12&9\\
		15&10&9&16
		\end{array}
		\right]
		\left[
		\begin{array}{c}
		y_{1}\\
		y_{2}\\
		y_{3}\\
		y_{4}
		\end{array}
		\right]=
		\left[
		\begin{array}{c}
		14\\
		10\\
		8\\
		11
		\end{array}
		\right].	
		\]
		The conditions of Theorem~\ref{psuedothm} hold for the system $ AA^{T}Y=b $, that is $ ((AA^{T})(AA^{T})^{-})_{ij} \leq b_{i} -b_{j} $ for any $ i, j \in \{ 1, \cdots, 4 \} $, where $ (AA^{T})(AA^{T})^{-} $ is the following matrix:
		\[
		\left[
		\begin{array}{cccc}
		0&4&6&-1\\
		-4&0&2&-5\\
		-6&-2&0&-7\\
		-9&-5&-3&0
		\end{array}
		\right].
		\]
		As such, the system $ AA^{T}Y=b $ has the maximal solution $ Y^{*}=(AA^{T})^{-}b $:
		\[
		Y^{*}=\left[
		\begin{array}{cccc}
		-24&-20&-18&-25\\
		-20&-16&-14&-21\\
		-18&-14&-12&-19\\
		-25&-21&-19&-16
		\end{array}
		\right]
		\left[
		\begin{array}{c}
		14\\
		10\\
		8\\
		11
		\end{array}
		\right]=
		\left[
		\begin{array}{c}
		-10\\
		-6\\
		-4\\
		-5
		\end{array}
		\right].	
		\]
		Hence, $ X=A^{T}Y^{*} $ is a solution of the non-square system $ AX=b $:
		\[
		X=\left[
		\begin{array}{c}
		-3\\
		3\\
		2\\
		-3\\
		-2
		\end{array}
		\right],	
		\]
		which is not necessarily maximal solution.
	\end{ex}
	\begin{ex}
		Let $ A \in M_{4 \times 3}(S) $. Consider the following non-square system $ AX=b $:
		\[
		\left[
		\begin{array}{ccc}
		2&5&-2\\
		1&4&3\\
		7&8&1\\
		0&1&4
		\end{array}
		\right]
		\left[
		\begin{array}{c}
		x_{1}\\
		x_{2}\\
		x_{3}
		\end{array}
		\right]=
		\left[
		\begin{array}{c}
		8\\
		3\\
		5\\
		2
		\end{array}
		\right].	
		\]
		According to the second case of the extension method
		, the following square system $ A^{T}AX=A^{T}b $ is obtained from  the non-square system\\ $ AX=b $:
		\[
		\left[
		\begin{array}{ccc}
		14&15&8\\
		15&16&9\\
		8&9&8
		\end{array}
		\right]
		\left[
		\begin{array}{c}
		x_{1}\\
		x_{2}\\
		x_{3}
		\end{array}
		\right]=
		\left[
		\begin{array}{c}
		12\\
		13\\
		6
		\end{array}
		\right].	
		\]
		Since the conditions $ ((A^{T}A)(A^{T}A)^{-})_{ij} \leq (A^{T}b)_{i} - (A^{T}b)_{j} $ hold for any $ i,j \in \{ 1,2,3 \} $, where $ (A^{T}A)(A^{T}A)^{-} $ is the following matrix:
		\[
		\left[
		\begin{array}{ccc}
		0&-1&0\\
		-1&0&-1\\
		-6&-7&0
		\end{array}
		\right],
		\]
		by Theorem~\ref{psuedothm}, the system $ A^{T}AX= A^{T}b $ has the maximal solution $ X^{*}= (A^{T}A)^{-}A^{T}b $:
		\[
		X^{*}=\left[
		\begin{array}{ccc}
		-14&-15&-14\\
		-15&-16&-15\\
		-14&-15&-8
		\end{array}
		\right]
		\left[
		\begin{array}{c}
		12\\
		13\\
		6
		\end{array}
		\right]=
		\left[
		\begin{array}{c}
		-2\\
		-3\\
		-2
		\end{array}
		\right],	
		\]
		while $ X^{*} $ is not a solution of the system $ AX=b $:
		\[
		AX^{*}=\left[
		\begin{array}{ccc}
		2&5&-2\\
		1&4&3\\
		7&8&1\\
		0&1&4
		\end{array}
		\right]
		\left[
		\begin{array}{c}
		-2\\
		-3\\
		-2
		\end{array}
		\right]=
		\left[
		\begin{array}{c}
		2\\
		-1\\
		5\\
		2
		\end{array}
		\right] \neq b	
		\]
		In this manner, we actually solve the nearest square system corresponding to the system $ AX= b $. 	
	\end{ex}
	\subsection{Extended Cramer's rule  for solving a linear system of equations}
	
	In classic linear algebra, Cramer's rule determines the unique solution of a linear system of equations, $ AX=b $, when $ A $ is an invertible matrix. Furthermore, Tan shows that an invertible matrix $ A $ over a commutative semiring has the inverse $ A^{-1}= det_{\varepsilon}(A)^{-1}adj_{\varepsilon}(A) $ (see \cite{T3}).\\
	In \cite{S1}, Sararnrakskul proves a square matrix $A$ over a semifield is invertible if and only if every row and every column of $A$ contains exactly one nonzero element. Moreover, in \cite{T3}, Tan develops Cramer's rule for a system $AX=b$ when $A$ is an invertible matrix over a commutative semiring. Consequently, using Cramer's rule for solving a system of linear equations over semifields and, as a special case, over tropical semirings, seems to be limited to matrices containing exactly one nonzero element in every row and every column.\\
	In the next theorem, we present an extended version of Cramer's rule to determine the maximal solution of the system $ AX=b $ by using the pseudo inverse of matrix $ A $, where $A$ is an arbitrary square matrix.
	\begin{theorem}\label{ECramer}
		Let $ A \in M_{n}(S) $, $ b\in S ^{n} $ be a regular vector, and $ det_{\varepsilon}(A) \in U(S)$. Then the system $ AX=b$ has the maximal solution $$ X^{*}=(d^{-1}d_{1}, d^{-1}d_{2}, \cdots, d^{-1}d_{n})^{T} $$ if and only if $ (AA^{-})_{ij} \leq b_{i} -b_{j} $ for every $ i,j \in \underline{n} $, where $ d= det_{\varepsilon}(A) $ and $d_{j} =  det_{\varepsilon}(A_{[j]})$ for $ j \in \underline{n} $ and $A_{[j]} $ is the matrix formed by replacing the $ j $-th column of $ A $ by the column vector $b $.
	\end{theorem}
	\begin{proof}
		By theorem~\ref{psuedothm}, the inequalities $ (AA^{-})_{ij} \leq b_{i} -b_{j} $ for every $i,j \in \underline{n}$  are equivalent to have the maximal solution $ X^{*}=A^{-}b$ for the system $ AX=b $, so 
		\begin{align*}
		X^{*} &=A^{-}b \\
		& =det_{\varepsilon}(A)^{-1} adj_{\varepsilon}(A)b\\
		& =det_{\varepsilon}(A)^{-1}
		\left[
		\begin{array}{ccc}
		A_{11}&\cdots&A_{1n}\\
		\vdots&\ddots&\vdots\\
		A_{n1}&\cdots&A_{nn}\\
		\end{array}
		\right]
		\left[
		\begin{array}{c}
		b_{1}\\
		\vdots\\
		b_{n}
		\end{array}
		\right]\\
		& =det_{\varepsilon}(A)^{-1}
		\left[
		\begin{array}{c}
		\max (A_{11}+ b_{1}, \cdots, A_{1n}+ b_{n})\\
		\vdots\\
		\max (A_{n1}+ b_{1}, \cdots, A_{nn}+ b_{n})
		\end{array}
		\right],	
		\end{align*}
		where $ A_{ij}= (adj_{\varepsilon}(A))_{ij}= det_{\varepsilon}(A(j|i)) $. The $ j $-th component of $ X^{*} $ is:
		\begin{align}
		x^{*}_{j} & \nonumber =det_{\varepsilon}(A)^{-1}\max (A_{j1}+ b_{1}, \cdots, A_{jn}+ b_{n})\\ \nonumber
		& =det_{\varepsilon}(A)^{-1} \max(det_{\varepsilon}(A(1|j))+ b_{1}, \cdots, det_{\varepsilon}(A(n|j))+ b_{n})\\ 
		& =det_{\varepsilon}(A)^{-1} det_{\varepsilon}(
		\left[
		\begin{array}{ccccccc}
		a_{11}&\cdots&a_{1(j-1)}&b_{1}&a_{1(j+1)}\cdots&a_{1n}\\
		\vdots&~&\vdots&\vdots&\vdots&\vdots&~\\
		a_{n1}&\cdots&a_{n(j-1)}&b_{n}&a_{n(j+1)}\cdots&a_{nn}
		\end{array}
		\right]
		)\\ \nonumber
		&\ = det_{\varepsilon}(A)^{-1} det_{\varepsilon}(A_{[j]}) \\ \nonumber
		&\ =d^{-1}d_{j},
		\end{align}
		i.e., $X^{*}=(d^{-1}d_{1},d^{-1}d_{2},\cdots,d^{-1}d_{n})^{T}$.
		Note that the equality $ (4.2) $ is obtained from Laplace's theorem for semirings (see Theorem~3.3 in~\cite{T1}).
	\end{proof}
	\begin{remark}
		In classic linear algebra, the unique solution of the system $ AX=b $, when $ det(A) \neq 0$, can be obtained from Cramer's rule without calculating the inverse matrix of $ A$. Similarly, in max-plus linear algebra, we can use the extended Cramer's rule to get the maximal solution of the system $ AX=b $, when $ (AA^{-})_{ij} \leq b_{i}- b_{j} $ for any $ i,j \in \underline{n}$, without computing $ A^{-}$ (see Remark~\ref{AA^{-}rem}).
	\end{remark}
	\begin{ex}
		Let $ A \in M_{3}(S)$. Consider the following system $ AX= b$:
		\[
		\left[
		\begin{array}{ccc}
		5&2&6\\
		4&1&4\\
		3&7&14
		\end{array}
		\right]
		\left[
		\begin{array}{c}
		x_{1}\\
		x_{2}\\
		x_{3}
		\end{array}
		\right]
		=
		\left[
		\begin{array}{c}
		5\\
		4\\
		13
		\end{array}
		\right].
		\]
		Clearly, $ (AA^{-})_{ij} \leq b_{i} -b_{j} \leq -(AA^{-})_{ji} $ hold for any $ 1 \leq i\leq j \leq 3 $, where $ AA^{-}$ is the following matrix:
		\[
		\left[
		\begin{array}{ccc}
		0&1&-8\\
		-1&0&-9\\
		5&6&0
		\end{array}
		\right]
		\]
		Due to Theorem~\ref{ECramer}, the system $ AX=b $ has the maximal solution $ X^{*}=(x^{*}_{i})_{i=1}^{3}$ where:
		\[
		x^{*}_{1}=det_{\varepsilon}(\left[
		\begin{array}{ccc}
		5&2&6\\
		4&1&4\\
		13&7&14
		\end{array}
		\right]) -det_{\varepsilon}(A)= 20 - 20 = 0,
		\]
		\[
		x^{*}_{2}=det_{\varepsilon}(\left[
		\begin{array}{ccc}
		5&5&6\\
		4&4&4\\
		3&13&14
		\end{array}
		\right]) -det_{\varepsilon}(A)= 23 - 20 = 3,
		\]
		\[
		x^{*}_{3}=det_{\varepsilon}(\left[
		\begin{array}{ccc}
		5&2&5\\
		4&1&4\\
		3&7&13
		\end{array}
		\right]) -det_{\varepsilon}(A)= 19 - 20 = -1.
		\]
	\end{ex}
	\section{Concluding Remarks}\label{remarks}
	In this paper, we studied the order reduction of  systems of linear equations through a row-column analysis technique over tropical semirings in order to simplify their solution process. Using the pseduo-inverse of the matrix of a linear system of equations, we also presented necessary and sufficient conditions for the system to have a maximal solution. We obtained this maximal solution through a new version of Cramer's rule as well.


	
	\newpage
	\appendix
	\begin{center}
		\begin{table}[!tb]
			\begin{verbatim}
				MaxPlusDet := proc (A::Matrix) 
				local i, j, s, n, detA, ind, K, V; 
				description "This program finds the determinant of a square matrix in max-plus."; 
				Use LinearAlgebra in
				n := ColumnDimension(A); 
				V := Matrix(n); 
				ind := Vector(n);
				if n = 1 then
				   V := A[1, 1];
				   detA := V;
				   ind[1] := 1
				elif n = 2 then
				   V[1, 1] := A[1, 1]+ A[2, 2];
				   V[1, 2] := A[1, 2]+ A[2, 1]; 
				   V[2, 1] := A[1, 2]+ A[2, 1];
				   V[2, 2] := A[1, 1]+ A[2, 2];
				   detA := max(V);
				   for s to 2 do 
				     K := V[s, 1 .. 2];
				     ind[s] := max[index](K)
				   end do;
				else
				   for i to n do
				     for j to n do
				       V[i, j] := A[i, j]+ op(1, MaxPlusDet(A[[1 .. i-1, i+1 .. n], [1 .. j-1, j+1 .. n]]));
				     end do; 
				     detA := max(V);
				     K := V[i, 1 .. n];
				     ind[i] := max[index](K);
				   end do;
				end if;
				end use:
				[detA, ind, V]			
				end proc:
			\end{verbatim}
			\caption{Finding the determinant of a square matrix in max-plus} \label{tab1}
		\end{table}
	\end{center}
	
	\begin{center}
		\begin{table}[!tb]
			\begin{verbatim}
				Matmul := proc (A::Matrix, B::Matrix) 
				local i, j, m, n, p, q, C, L;
				description "This program finds the multiplication of two matrices in max-plus.";\\
				Use LinearAlgebra in
				m := RowDimension(A);
				n := ColumnDimension(A);
				p := RowDimension(B);
				q := ColumnDimension(B);
				C := Matrix(m, q); 
				if n <> p then
				   print('impossible');
				   break
				else
				   for i to m do
				     for j to q do
				       L := [seq(A[i, k]+B[k, j], k = 1 .. n)];
				       C[i, j] := max(L)
				     end do
				   end do
				end if;
				end use:
				C
				end proc:
			\end{verbatim}
			\caption{Calculation of matrix multiplication in max-plus} \label{tab3}
		\end{table}
	\end{center}
	
	\begin{center}
		\begin{table}[!tb]
			\begin{verbatim}
				Ainv := proc (A::Matrix)
				local n, d, H, V, B, C, E, G, Z, i, j;
				description "This program finds a pseudo-inverse of A in max-plus.";
				use LinearAlgebra in
				n := ColumnDimension(A);
				d := op(1, MaxPlusDet(A));
				H := Matrix(1 .. n, 1 .. n, d);
				Z := Matrix(1 .. n, 1 .. n);
				for i to n do
				   for j to n do
				     if A[i, j] = (-1)*Float(infinity) then
				       Z[i, j] := Float(infinity);
				     else
				       Z[i, j] := A[i, j];
				     end if:
				   end do:
				end do:
				V := op(3, MaxPlusDet(A));
				B := V-H-Z;
				G := Transpose(B);
				C := Matmul(A, G);
				E := Matmul(G, A);
				end use:
				G, C
				end proc:
			\end{verbatim}
			\caption{Calculating the pseduo-inverse, $A^{-}$, of a square matrix, $A$, as well as $AA^{-}$ in max-plus} \label{tab4}
		\end{table}
	\end{center}
	

\begin{thebibliography}{33}
		\bibitem{A1}
		L. Aceto, Z. Esik, A. Ingolfsdottir, \textit{Equational theories of tropical semirings.} Theoretical Computer Science \textbf{298(3)} (2003), 417--469.
		\bibitem{G1}
		J. S. Golan, \textit{Semirings and their Applications.} Springer Science Business Media, 2013.
		\bibitem{G2}
		M. Gondran, M. Minoux, \textit{Graphs, dioids and semirings new models and algorithms.} Springer Science Business Media, 2008.
		\bibitem{K1}
		N. Krivulin, \textit{Complete algebraic solution of multidimensional optimization problems in tropical semifield.} Journal of logical and algebraic methods in programming \textbf{99} (2018), 26--40.
		\bibitem{K2}
		N. Krivulin, S. Sergeev, \textit{Tropical optimization techniques in multi-criteria decision making with Analytical Hierarchy Process.} In 2017 European Modelling Symposium (EMS) (2017), 38--43. IEEE.
		\bibitem{M1}
		W. M. McEneaney, \textit{Max-plus methods for nonlinear control and estimation.} Springer Science Business Media, 2006.
		\bibitem{C1}
		G. J. Olsder, C. Roos, \textit{Cramer and Cayley-Hamilton in the max algebra.} Linear Algebra and its Applications \textbf{101} (1988), 87--108.
		\bibitem{S1}
		R. I. Sararnrakskul, S. Sombatboriboon, P. Lertwichitsilp, \textit{Invertible Matrices over semifields.} East-West Journal of Mathematics \textbf{12(1)} (2010).
		\bibitem{T1}
		Y. J. Tan, \textit{Determinants of matrices over semirings.} Linear and Multilinear Algebra \textbf{62(4)} (2014), 498--517.
		\bibitem{T2}
		Y. J. Tan, \textit{Inner products on semimodules over a commutative semiring.} Linear Algebra and its Applications. \textbf{460} (2014), 151--173.
		\bibitem{T3}
		Y. J. Tan, \textit{On invertible matrices over commutative semirings.} Linear and Multilinear Algebra \textbf{61(6)} (2013), 710--724.
		\bibitem{V1}
		H. S. Vandiver, \textit{Note on a simple type of algebra in which the cancellation law of addition does not hold.} Bulletin of the American Mathematical Society. \textbf{40(12)} (1934), 914--920.
		\bibitem{W1}
		D. Wilding, \textit{Linear algebra over semirings.} (Doctoral dissertation, The University of Manchester (United Kingdom), 2015.	
	\end{thebibliography}
\end{document}